\documentclass[12pt]{article}
\usepackage{amsfonts}
\usepackage{amsmath,amsthm}
\usepackage{enumerate}
\usepackage{geometry}
\usepackage{setspace}
\usepackage{comment}
\usepackage{appendix}
\numberwithin{equation}{section}

\newtheorem{theorem}{Theorem}[section]
\newtheorem{lemma}[theorem]{Lemma}
\newtheorem{proposition}[theorem]{Proposition}
\newtheorem{corollary}[theorem]{Corollary}

\newtheorem{conjecture}[theorem]{Conjecture}

\theoremstyle{definition}

\DeclareMathOperator{\Tr}{Tr}
\DeclareMathOperator{\End}{End}
\DeclareMathOperator{\wtr}{wt}
\newcommand{\qd}{\ensuremath{\partial}}
\newcommand{\zh}{\Theta}
\newcommand{\wtil}{\ensuremath{\tilde{\omega}}}
\newcommand{\der}[1]{\ensuremath{\partial_{#1}}}

\newcommand{\vv}{\mathbf{1}}
\newcommand{\tl}{\widetilde{\lambda}}

\begin{document}

\title{On the Torus Degeneration of the Genus Two Partition Function}
\author{
Donny Hurley\thanks{
Supported by the Science Foundation Ireland Frontiers of Research Programme } 
\ and Michael P. Tuite \\
%EndAName
School of Mathematics, Statistics and Applied Mathematics \\
National University of Ireland Galway, \\
University Road,
Galway, Ireland.}
\maketitle

\begin{abstract}
We consider the partition function of a general vertex operator algebra $V$ on a genus two Riemann surface formed by  sewing together two tori. We consider the non-trivial degeneration limit where one torus is pinched down to a Riemann sphere and show that the genus one partition function on the degenerate torus is recovered up to an explicit universal $V$-independent multiplicative factor raised to the power of the central charge.
\end{abstract}

\newpage

\section{Introduction}
\label{Intro}
A Vertex Operator Algebra (VOA) (e.g. \cite{FLM,K}) is an algebraic structure closely related to Conformal Field Theory (CFT) in physics e.g.~\cite{BPZ,DMS}.  There is a well developed mathematical theory of partition and $n$-point correlation functions for a VOA associated with a genus one torus \cite{Z}. More recently, a study has 
begun of VOA partition and correlation functions on a genus two Riemann surface formed from one or two sewn tori where the genus two partition and correlation functions are defined in terms of correlation function data on the genus one surface(s) \cite{T,MT1,MT2,MT3}.
In particular, in this paper we consider a genus two Riemann surface formed by sewing two tori and study the behavior of the genus two partition function for a general VOA $V$ in the non-trivial torus degeneration limit where one of the sewn tori is pinched down to a Riemann sphere. 
We show that the genus one partition function on the degenerate torus is recovered up to an explicit universal multiplicative factor raised to the power of the central charge. 
This universal factor is independent of the VOA $V$ and is determined explicitly in terms of the genus two partition function for the Heisenberg VOA.
This provides a non-trivial test on a more general conjecture concerning the independence of the genus two partition on the sewing scheme up to a general universal $V$-independent factor.

The paper begins in section~\ref{sec:VOAs} with a brief review of the theory of Vertex Operator Algebras (VOA) including Zhu theory for genus one partition and 1-point functions. In particular we describe a general expression for the genus one 1-point function for an arbitrary Virasoro descendent of the vacuum.
In section~\ref{sec:g2} we review the procedure for sewing two tori to obtain a genus two Riemann surface. We discuss the torus degeneration limit where one of the sewn tori is pinched down to a Riemann sphere. In particular, we describe the modular parameter $\tau$ on the degenerate surface (which is given by a limit of a diagonal component of the genus two period matrix) in terms of the sewing parameters.  
We then review the definition of the genus two partition function for a VOA and, in particular,  the explicit form of the partition function for the Heisenberg VOA and its modules since these expressions play a crucial role in our later considerations. 
Our main result is Theorem~\ref{g2_mainresult} wherein we determine the genus one degeneration limit of the genus two partition function of an arbitrary VOA $V$  with a genus two partition function. 
We show that the degeneration limit is described in terms of genus one 1-point correlation functions for a special set of Virasoro vacuum descendent vectors associated with the exponential conformal map of genus one Zhu theory. This result allows us to show that the genus one partition function on the degenerate torus is recovered up to an explicit $V$-independent  universal multiplicative factor raised to the power of the central charge. 
In the appendix we discuss some properties of the Virasoro vacuum descendent vectors in more detail.

\section{Vertex Operator Algebras}
\label{sec:VOAs}
\subsection{VOA Axioms} 
We  review some basic properties of Vertex Operator Algebras  e.g. \cite{FLM,FHL,K,LL,MN,MT4}.
A Vertex Operator Algebra (VOA) is a quadruple $(V, Y, \vv, \omega)$
consisting of a $\mathbb{Z}$-graded complex vector space $V =\bigoplus_{n\in\mathbb{Z}}V_n$ with $\dim V_{n}<\infty$, a linear map $Y \rightarrow \End V [z, z^{-1}]$, for formal parameter $z$ and a pair of states: the vacuum $\mathbf{1} \in V_0$ and a conformal vector $\omega \in V_2$. For each $u \in V$ we have a vertex operator
\begin{equation*}
Y(u,z) = \sum_{n\in\mathbb{Z}}u_n z^{-n-1},
\end{equation*}
with $u_n\in \End V$ which satisfies the following axioms: 
\begin{itemize}
\item \textbf{Locality.}
$(z_1-z_2)^N [Y(u,z_1),Y(v,z_2)] = 0$ for each $u,v\in V$ for some integer $N\gg 0$.
\item \textbf{Creativity.} 
$Y(u,z)\vv = u + O(z)$.
\item \textbf{Virasoro Structure.}
For the  conformal vector $\omega\in V_2$ 
\begin{displaymath}
Y(\omega,z) = \sum_{n\in\mathbb{Z}}L_{n} z^{-n-2},
\end{displaymath}
where the modes $L_{n}=\omega_{n+1}$ satisfy a Virasoro Algebra of central charge $C$ with bracket relation
\begin{align*}
[L_{m},L_{n}] =(m - n)L_{m+n} + 
C\frac{m^3 - m}{12}\operatorname{Id}_{V}\delta_{m,-n}.
\end{align*} 
Furthermore, the $\mathbb{Z}$-grading on $V$ is provided by $L_{0}$ with $L_{0}u=nu$ for all $u\in V_{n}$ where $n$ is referred to as the weight $\wtr(u)$ of $u$. 
\item \textbf{Translation.}
$[L_{-1}, Y(u, z)] = \partial_z Y(u, z)$.
\end{itemize}
\medskip

\noindent Amongst many things, these axioms imply that 
\begin{equation}
u_n:V_k\rightarrow V_{k-n+\wtr(u)-1},
\label{un}
\end{equation}
for $u$ of weight $\wtr(u)$. In particular,  define
\begin{equation}
o(u)=u_{\wtr(u)-1}:V_k\rightarrow V_{k},
\label{ou}
\end{equation}
which is extended by linearity to all $u$.
\begin{comment}
\begin{eqnarray}
(x+z)^N Y(u,x+z)Y(v,z)w &=& (x+z)^N Y(Y(u,x)v,z)w,
\label{Assoc}
\end{eqnarray}
for $u,v,w\in V$ and integer $N \gg 0$  \cite{FHL}. 
\end{comment}
In this paper we consider VOAs of CFT-type for which $V_{0}=\mathbb{C}\vv$ (and consequently, $V_{n}=0$ for $n<0$ \cite{MT2}). 
\medskip

A bilinear form $\langle \ ,\rangle :V\times V{\longrightarrow }\mathbb{C}$
is called invariant if for all $u,a,b\in V$  \cite{FHL} 
\begin{equation*}
\langle Y(u,z)a,b\rangle =\langle a,Y^{\dagger }(u,z)b\rangle,
\end{equation*}%
with adjoint vertex operator defined by  
\begin{equation}
Y^{\dagger }(u,z)=\sum_{n\in\mathbb{Z}}u^{\dagger }_n z^{-n-1}
=Y\left(e^{z L_1}\left(-z^{-2}\right)^{L_0}u,z^{-1}\right).  \label{intro_adjop}
\end{equation}
We note that $L^{\dagger }_n=L_{-n}$,  
the bilinear form is  symmetric and 
$\langle a,b\rangle =0$ for homogeneous $a$ and $b$ with $\wtr(a)\neq \wtr(b)$ \cite{FHL}. 
A VOA $V$ of CFT-type has a unique non-zero invariant bilinear form up to a scalar \cite{Li} if $V$ is self-dual ($V$ is isomorphic to the dual module $V^{\prime}$ as a $V$-module \cite{FHL}). Furthermore, if $V$ is simple and we normalise so that $\langle \vv,\vv\rangle  = 1$ then such a form is unique and non-degenerate. We refer to this non-degenerate bilinear form as the Li-Zamolodchikov metric on $V$ \cite{MT2, MT3}.

\subsection{Modular Forms}
We review some aspects of the theory of modular functions  e.g.\ \cite{S}. Define the Eisenstein series for even $k\ge 2$ by
\begin{equation}
E_{k}(q)=-\frac{B_{k}}{k!}+\frac{2}{(k-1)!}\sum\limits_{n\geq 1}\frac{n^{k-1}q^{n}}{1-q^{n}},  \label{intro_Ekdefn}
\end{equation}%
where $B_k$ is the $k$th Bernoulli number defined by 
\begin{equation}
\label{bernoullidef}
\frac{z}{e^z-1}= \sum\limits_{k\geq 0} \frac{B_{k}}{k!}z^k = 1-\frac{1}{2}z+\frac{1}{12}z^2+ \hdots,
\end{equation}
and $E_{k}(q)=0$  for $k$ odd. For $q=e^{2\pi i\tau}$ with $\tau\in\mathbb{H}$, the complex upper half plane, $E_k(q)$ is a holomorphic modular form of weight $k$ on $SL(2, \mathbb{Z})$ for $k\ge 4$. $E_2(q)$ is a quasi-modular form of weight $2$ and the ring of quasi-modular forms is generated by $E_2(q),E_4(q)$ and $E_6(q)$ \cite{KZ}. 

Define the differential operator
\begin{equation*}
\qd \equiv q\frac{\partial}{\partial q}=\frac{1}{2\pi i }\frac{\partial}{\partial \tau}.
\end{equation*}
We recall that for a modular form $f_k(q)$ of weight $k$, the Serre modular derivative $(\qd +k E_2(q))f_k(q)$ is a modular form of weight $k+2$. Furthermore, 
\begin{equation}
\qd E_2(q)=5E_4(q)-E_2(q)^2. \label{delE2}
\end{equation} 
Thus, for a  quasi-modular form $f_k(q)$ of weight $k$ then  $\qd f_k(q)$ is a quasi-modular form of weight $k+2$.  
Finally, we define the Dedekind $\eta$-function
\begin{eqnarray}
\eta(q) &=& q^{1/24}\prod_{n \geq 1}(1-q^n),
\label{intro_dedekindef}
\end{eqnarray}
and note that 
\begin{equation}
\label{deleta}
\qd \eta(q)=-\frac{1}{2}E_2(q)\eta(q).
\end{equation} 
\medskip

\subsection{The Square Bracket VOA}\label{sec:g1_kqbracket}
In order to define genus one $n$-point correlation functions, 
Zhu \cite{Z} introduced the so-called square bracket formalism. 
This consists of a second VOA $(V,Y[\ ,\ ],\mathbf{1},\widetilde{\omega})$ isomorphic to $(V,Y(\ ,\ ),\mathbf{1},\omega )$ with vertex operators
\begin{displaymath}
Y[v,z] = \sum_{n\in\mathbb{Z}}v[n]z^{-n-1} = Y\left (e^{zL_{0}}v,e^z-1\right).
\end{displaymath}
The new conformal vector is $\wtil=\omega-\frac{C}{24}\vv$ with vertex operator
\begin{align*}
Y[\widetilde{\omega},z] &= \sum_{n\in\mathbb{Z}}L[n]z^{-n-2}.
\end{align*}
The isomorphism between $(V,Y(\ ,\ ),\mathbf{1},\omega )$ and the square bracket VOA $\left (V,Y[\ ,\ ],\mathbf{1},\tilde{\omega}\right)$
is described by the linear operator
\begin{equation*}
T =\exp\left (\sum_{i \ge 0} \alpha_i L_i\right),
\end{equation*}
constructed from the exponential map  
\begin{align*}
\phi(z)&=e^z-1=\exp\left(\sum_{i=1}^\infty \alpha_i z^{i+1}\der{z}\right)z,
\end{align*}
for specific rational parameters $ \alpha_1=\frac{1}{2}, \alpha_2=-\frac{1}{12}, \alpha_3=\frac{1}{48},\ldots$  \cite{Z}.
The bracket square vertex operator can be also expressed by  
\begin{equation}
Y[a,z] = T  Y\left (T ^{-1} a,z\right)T ^{-1}. \label{g2_zhushowsthat}
\end{equation}
For example, $\wtil = T \omega $ with $L[n]=T  L_n T ^{-1} $.  Thus $T  p=p$ for a 
primary vector $p$ (i.e. $L_k p=0$ for all $k>0$). 
For a general Virasoro descendent $v=L_{-k_1}\ldots L_{-k_k}p$, we find
\begin{displaymath}
T  v=L[-k_1]\ldots L[-k_k]p.
\end{displaymath}

$V$ has  decomposition $V =\bigoplus_{n\in\mathbb{Z}}V_{[n]}$ where $L[0]u=nu$ for $u\in V_{[n]}$ for square bracket weight $\wtr[u]=n$.
Similarly, we can define a square bracket Li-Zamolodchikov metric $\langle \ ,\rangle_{[\,]}$ where
\begin{equation*}
\left\langle Y[u,z]a,b\right\rangle_{[\,]} =\left\langle a,Y^{\dagger }[u,z]b\right\rangle_{[\,]},
\end{equation*}
for all $a,b\in V$ with $Y^{\dagger}[u,z]=Y\left[e^{z L[1]}\left(-{z^{-2}}\right)^{L[0]}u,z^{-1}\right]$. 
It follows from \eqref{g2_zhushowsthat} that the original and square bracket Li-Zamolodchikov metrics are related by
\begin{equation}
\left\langle a,b\right\rangle_{(\,)}=\left\langle T  a,T  b\right\rangle_{[\,]} . \label{LiZrsq}
\end{equation}

From the definition of the adjoint operator \eqref{intro_adjop} we have
\begin{equation*}
T ^\dagger = \exp\left(\sum_{i\ge 0}\alpha_iL_{-i}\right).
\end{equation*}
We define
\begin{align}
\lambda &=\sum_{n\ge 0}\lambda^{(n)} = T^\dagger\vv \label{g2_lambdadefneqn},\\
\tl &= \sum_{n\ge 0}\tl^{[n]}= T \lambda , \label{g2_tldefneqn}
\end{align}
for round and square bracket Virasoro vacuum descendents $\lambda^{(n)}\in V_n$ and $\tl^{[n]}=T\lambda^{(n)}\in V_{[n]}$.
In Proposition~\eqref{prop:lambda} in the Appendix we show that $\lambda^{(n)}= 0$ for $n$ odd with 
\begin{align*}
\lambda^{(0)} &= \vv,\qquad \lambda^{(2)} = -\frac{1}{12}L_{-2}\vv, \qquad
\lambda^{(4)} = \frac{1}{288}L^2_{-2}\vv - \frac{1}{480}L_{-4}\vv,\\
\lambda^{(6)} &= -\frac{1}{10368}L^{3}_{-2}\vv + \frac{1}{5760}L_{-4}L_{-2}\vv + \frac{1}{12096}L_{-6}\vv,
\end{align*}
with similar formulas for $\tl^{[n]}$ in terms of square bracket Virasoro descendents.
 
Finally, using  \eqref{ou} and \eqref{LiZrsq} we note that
\begin{lemma}\label{g2_kqbrackisolemma}
For $u\in V_{[n]}$ then
\begin{equation*}
\langle\vv,o(u)\vv\rangle_{(\,)} = \langle\tl^{[n]},u\rangle_{[\,]}.
\end{equation*}
\end{lemma}
\begin{proof}
Using the creation axiom we have $\langle \vv,o(u)\vv\rangle_{(\,)} = \langle \vv,u\rangle_{(\,)}$. 
Thus
\begin{align*}
&\langle \vv,o(u)\vv\rangle_{(\,)} = \langle \vv, T  T^{-1}u\rangle_{(\,)}
=\langle \lambda , T^{-1}u\rangle_{(\,)}.
\end{align*}
Applying  \eqref{LiZrsq} implies   since $u\in V_{[n]}$ that
\begin{align*}
\langle \lambda , T^{-1}u\rangle_{(\,)}&= \langle \tl, u\rangle_{[\,]}= \langle \tl^{[n]}, u\rangle_{[\,]}.
\end{align*}

\end{proof}

\medskip

\subsection{Genus One Partition and Correlation Functions}\label{sec:g1_pfun}
The genus one partition function for $V$ is defined by 
\begin{equation}
Z_{V}^{(1)}(q)=\Tr_{V}\left(q^{L_{0}-C/24}\right),
\label{Z1_1}
\end{equation}
and the 1-point correlation function by 
\begin{equation}
Z_{V}^{(1)}(u,q)=\Tr_{V}\left(o(u)q^{L_{0}-C/24}\right),
\label{Zu_1} 
\end{equation}
recalling \eqref{ou}. We may similarly define the partition function and 1-point function for any module $N$ of $V$ \cite{FHL} by replacing the trace over $V$  by a trace over $N$. Thus, as is well known, for the rank 1 Heisenberg VOA $M$ we find $Z_{M}^{(1)}(q)=1/\eta(q)$, whereas for an irreducible $M$-module $N_\alpha=M\otimes e^\alpha$ we find $Z_{N_\alpha}^{(1)}(q)=q^{\alpha^2/2}/\eta(q)$.

Zhu obtained various recursion formulas for $n$-point functions \cite{Z}. In particular 
\begin{theorem}[Zhu]
\label{thm:Zhured}
For a $V$-module $N$, the 1-point function for the Virasoro descendent $L[-k]u$ of $u\in V$ for $k\ge 1$ obeys
\begin{equation} 
Z_N^{(1)}(L[-k]u,q) = \delta_{k,2}\qd Z_N^{(1)}(q) + \sum_{r\ge 0}(-1)^r\binom{k+r-1}{r+1} E_{k+r}(q)Z_N^{(1)}(L[r]u,q).
\end{equation}
\label{Zred}
\end{theorem} 
%\noindent Notice that the sum is finite since $L[r]u=0$ for $r\gg 0$. 
We note the following useful consequence of Theorem~\ref{thm:Zhured} concerning the Virasoro descendents of a primary vector $p$:
\begin{proposition}\label{g1_primtheorem}
Let $v=L[-k_1]\ldots L[-k_m]p$ be a  Virasoro descendent of a primary vector $p$ of weight $\wtr[p]$.
The one point function for $v$ for a $V$-module $N$ is given by 
\begin{equation}
Z_N^{(1)}(v,q) = \sum_{i=0}^m F_{i}^m(q,C)\qd^{i}Z_N^{(1)}(p,q),
\label{FqCsum}
\end{equation}
where $F_i(q,C)$ is a quasi-modular form of weight $k_1+\ldots+ k_m-2i$. Furthermore, 
\begin{equation*}
 F_i^m(q,C)= \sum_{j=0}^{\lfloor\frac{1}{2}(m-i)\rfloor} F_{ij}^m(q)C^j,
\end{equation*} 
where $F_{ij}^m(q)$ may depend on $\wtr[p]$ but is independent of the VOA $V$.
\end{proposition}
\begin{proof}
We prove the result by induction in $m$. For $m=1$ we use \eqref{Zred} to find
\begin{align*}
Z_N^{(1)}(L[-k]p,q)&=\delta_{k,2}\qd Z_N^{(1)}(p,q) + (k-1)\wtr[p]E_k(q) Z_N^{(1)}(p,q),
\end{align*}
since $p$ is primary of weight $\wtr[p]$. Thus we obtain non-zero coefficients
\begin{align*}
F_0^1&=F_{00}^1= (k-1)\wtr[p]E_k(q),\\
 F_1^1&=F_{10}^1 = \delta_{k,2},
\end{align*}
which are quasi-modular forms of weight $k=k-2(0)$ and $0=k-2(1)$ (for $k=2$) respectively but are independent of the VOA $V$.

Consider the 1-point function for $L[-k]v$ with $v=L[-k_1]\ldots L[-k_m]p$ and apply \eqref{Zred} to obtain
\begin{equation}
Z_N^{(1)}(L[-k]v,q) = \delta_{k,2}\qd Z_N^{(1)}(v,q) + \sum_{r\ge0}(-1)^r\binom{k+r-1}{r+1}E_{k+r}(q)Z_N^{(1)}(L[r]v,q).\label{expansionm}
\end{equation}
By induction, we apply \eqref{FqCsum} to $Z_N^{(1)}(v,q)$ so that
\begin{align*}
\qd Z_N^{(1)}(v,q) &=\sum_{i=0}^m\qd F_i^m(q,C)\qd^iZ_N(v,q) + \sum_{i=1}^{m+1}F_{i-1}^m(q,C)\qd^{i}Z_N^{(1)}(v,q).
\end{align*}
$\qd F_i^m(q,C)$  and $F_{i-1}^m(q,C)$ are quasi-modular forms of weight $2+k_1+\ldots k_m-2i $ and $2+k_1+\ldots k_m-2(i-1) $ respectively as required. In both cases the highest power of $C$ is $\lfloor\frac{m-i}{2}\rfloor \leq \lfloor\frac{m+1-i}{2}\rfloor$. 

Consider the remaining $E_{k+r}(q)Z_N^{(1)}(L[r]v,q)$ terms in \eqref{expansionm}. The Virasoro algebra implies
\begin{align*}
L[r]v &= \sum_{s=1}^{m}L[-k_1]\ldots \left((r+k_s)L[r-k_s]+C\frac{k_s^3-k_s}{12}\delta_{r,k_s}  \right)\ldots L[-k_m]p.
\end{align*}
By induction, we may apply \eqref{FqCsum} to the various Virasoro descendents of $p$ appearing above to find $Z_N^{(1)}(L[r]v,q)$. In particular, it is easy to check that the coefficient of $\qd^{i}Z_N^{(1)}(v,q)$ in $Z_N^{(1)}(L[r]v,q)$ is quasi-primary of weight $k_1+\ldots+ k_m-r-2i$ so that the
coefficient of $\qd^{i}Z_N^{(1)}(v,q)$ in $Z_N^{(1)}(L[-k]v,q)$ is quasi-primary of weight $k+k_1+\ldots+ k_m-2i$, as required.
Finally, if $r=k_s$ for some $k_s$, then a contribution arises with highest power of $C$ given by $1+\lfloor\frac{m-1-i}{2}\rfloor=\lfloor\frac{m+1-i}{2}\rfloor$. 
\end{proof}

For our later purposes, it is useful to define a normalized partition function
\begin{equation}
\zh_N^{(1)}(q) = \eta^C(q) Z_N^{(1)}(q),
\label{thetaN}
\end{equation} 
 for a $V$-module $N$.
We then find that Proposition~\ref{g1_primtheorem} implies 
\begin{corollary}\label{g1_maintheoremcorol}
Let $v=L[-k_1]\ldots L[-k_m]\vv$ be a Virasoro descendent of the vacuum of weight $n=k_1+\ldots+k_m$. Then  the one point function for a $V$-module $N$ is given by 
\begin{equation*}
Z_N^{(1)}(L[-k_1]\ldots L[-k_m]\vv,q) = \frac {1}{\eta(q)^C}\sum_{i=0}^m\sum_{j=0}^{m-i} G^n_{ij}(q)C^j\qd^{i}\zh_N^{(1)}(q),
\end{equation*}
where $G^n_{ij}(q)$ is a quasi-modular form of weight $n-2i$ that is independent of the VOA and the central charge $C$.
\end{corollary}
\begin{proof}
The proof follows directly from Proposition~\ref{g1_maintheoremcorol} and \eqref{deleta}.
Alternatively, one can employ an explicit description of all Virasoro vacuum descendent correlation functions given in Theorem~3.8 of \cite{HT}.
\end{proof}

\medskip
\section{The Genus Two Partition Function}
\label{sec:g2}
\subsection{Sewing Two Tori}
We now review a general method due to Yamada \cite{Y} and developed further in \cite{MT1} 
for calculating the period matrix (and other structures) on a Riemann surface formed by sewing together two lower genus Riemann surfaces. In particular, we wish to describe the period matrix $\Omega _{ij} \in \mathbb{H}_{2}$, the genus $2$ Siegel complex upper half-space, on a genus two Riemann surface formed by sewing together two tori. Consider an oriented torus $\mathcal{S}_{a}=\mathbb{C}/\Lambda _{\tau _{a}}$ with lattice $\Lambda _{\tau _{a}}$ with basis $(2\pi i,2\pi i\tau _{a})$ for $\tau _{a}\in \mathbb{H}$, the complex upper half plane. For local coordinate $z_{a}\in \mathcal{S}_{a}$ consider the closed disk $\left\vert z_{a}\right\vert \leq r_{a}$. This is contained in $\mathcal{S}_{a}$ provided\ $r_{a}<\frac{1}{2}D(q_{a})$ where 
\begin{equation*}
D(q_{a})=\min_{\lambda \in \Lambda _{\tau _{a}},\lambda \neq 0}|\lambda |,
\end{equation*}%
is the minimal lattice distance.
\medskip
Introduce a sewing parameter $\epsilon \in \mathbb{C}$ where $|\epsilon|\leq r_{1}r_{2}<\frac{1}{4}D(q_{1})D(q_{2})$ and excise the disk $
\{z_{a},\left\vert z_{a}\right\vert \leq |\epsilon |/r_{\bar{a}}\}$ centred at $z_{a}=0$ to form a punctured torus 
\begin{equation*}
\hat{\mathcal{S}}_{a}=\mathcal{S}_{a}\backslash \{z_{a},\left\vert
z_{a}\right\vert \leq |\epsilon |/r_{\bar{a}}\},
\end{equation*}
where we use the convention 
\begin{equation*}
\overline{1}=2,\quad \overline{2}=1.
\end{equation*}%
Define the annulus
\begin{equation*}
\mathcal{A}_{a}=\{z_{a},|\epsilon |/r_{\bar{a}}\leq \left\vert z_{a}\right\vert \leq r_{a}\}\subset \hat{\mathcal{S}}_{a}.
\end{equation*}
We then identify $\mathcal{A}_{1}$ with $\mathcal{A}_{2}$ via the sewing relation 
\begin{equation*}
z_{1}z_{2}=\epsilon ,
\end{equation*}
to obtain an explicit construction of a genus two Riemann surface
\begin{equation*}
\mathcal{S}^{(2)}=\hat{\mathcal{S}}_{1}\cup \hat{\mathcal{S}}_{2}\cup (\mathcal{A}_{1}\simeq \mathcal{A}_{2}),
\end{equation*}
parametrised by the domain 
\begin{equation*}
\mathcal{D}^{\epsilon }=\{(\tau _{1},\tau _{2},\epsilon )\in \mathbb{H}\times\mathbb{ H}\times\mathbb{ C}\, :\, |\epsilon |<\frac{1}{4}D(q_{1})D(q_{2})\}. 
\end{equation*}
In \cite{Y}, Yamada describes a general method for computing
the period matrix on the sewn Riemann surface $\mathcal{S}^{(2)}$ in terms of data obtained from the two tori. This is described in detail in \cite{MT1} where the explicit form for $\Omega $ is obtained in terms of the infinite matrix $A_{a}=(A_{a}(k,l,q_{a},\epsilon ))$ for $k,l\geq 1$ where 
\begin{equation*}
A_{a}(k,l,q_{a},\epsilon)=\epsilon ^{(k+l)/2 }\frac{(-1)^{l+1}(k+l-1)!}{\sqrt{kl}(k-1)!(l-1)!}E_{k+l}(q_a ), 
\end{equation*}
where $E_{k+l}(q_a)$ is the Eisenstein series \eqref{intro_Ekdefn}. 

The matrix $I-A_{1}A_{2}$ where $I$ denotes the infinite identity matrix plays an important role in our discussion. In particular we have the following convergence theorem:
\begin{theorem}\cite{MT1}
\label{g2_thmA1A2} \ \ 
\begin{enumerate}
\item[(a)] $(I-A_{1}A_{2})^{-1} =I+\sum_{n\geq 1}(A_{1}A_{2})^{n}$ is convergent on $\mathcal{D}^{\epsilon
} $.
\item[(b)] $\det (I-A_{1}A_{2})$ defined by
\begin{align*}
\log \det (I-A_{1}A_{2}) =\Tr\log (I-A_{1}A_{2})=-\sum_{n\geq 1}\frac{1}{n}\Tr (A_{1}A_{2})^{n},
\end{align*}
 is holomorphic and non-vanishing  on $\mathcal{D}^{\epsilon }$.
\end{enumerate}
\end{theorem}
We also have an explicit expression for the period
matrix:
\begin{theorem}\cite{MT1}
\label{g2_thm_period_eps} The genus two period matrix $\Omega$ is holomorphic on $\mathcal{D}^{\epsilon }$ and is given by 
\begin{align*}
\Omega _{11} &=\tau _{1}+\frac{\epsilon}{2\pi i}
\left (A_{2}(I-A_{1}A_{2})^{-1}\right)(1,1), \\
\Omega _{22} &=\tau _{2}+\frac{\epsilon}{2\pi i}
\left (A_{1}(I-A_{2}A_{1})^{-1}\right)(1,1), \\
\Omega _{12} &=-\frac{\epsilon}{2\pi i} \left (I-A_{1}A_{2}\right)^{-1}(1,1),
\end{align*}%
where $(1,1)$ refers to the $(1,1)$-entry of the given matrix. 
\end{theorem}

The sewing scheme described above provides, by definition, a description of the degeneration of the Riemann surface to two tori given by the limit $\epsilon\rightarrow 0$ for fixed $q_1,q_2$. 
Of central interest in this paper is the degeneration to a torus where one of the sewn tori is pinched down to a Riemann sphere \cite{F}. 
In particular, pinching down the right torus corresponds to the limit $q_2 \rightarrow 0$ (or  equivalently $\tau_2 \rightarrow i\infty$) for fixed $q_1, \epsilon$ for which $A_{2} \rightarrow A_{2}(0) $ with 
\begin{equation}
\label{A20}
A_{2}(0) = \left(\frac{(-1)^{l}\epsilon ^{(k+l)/2}B_{k+l}}{\sqrt{kl}(k+l)(k-1)!(l-1)!}\right), 
\end{equation} 
for $k,l\geq 1$ and $B_{k+l}$ is the $(k+l)^{th}$ Bernoulli number  \eqref{bernoullidef}. The modular parameter on the degenerate torus, which we denote by $\tau$, is determined by \cite{F}
\begin{align}
\tau &=\lim_{q_2 \rightarrow 0}\Omega_{11}(q_{1},q_{2},\epsilon ) = \tau _{1}+\frac{\epsilon}{2\pi i}
\left (A_{2}(0)(I-A_{1}A_{2}(0))^{-1}\right)(1,1) \notag \\
&=
\tau_{{1}}+\frac{1}{2\pi   i}\left(-\frac{1}{12} {\epsilon}^{2}
+{\frac {1}{144}} {\epsilon}^{4}E_{{2}} \left( q_{{1}} \right) 
 \right)+O\left({\epsilon}^{6} \right).
\label{eq:tau}
\end{align}
There is a similar formula for the modular parameter of the degenerate torus corresponding to the limit  $q_1 \rightarrow 0$ for fixed $q_2, \epsilon$.

\subsection{The Genus Two Partition Function}
Let $V$ be a simple, self-dual VOA of CFT-type with a unique invertible Li-Zamolodchikov metric. We define the genus two partition function
associated with the above sewing scheme by \cite{T,MT2}
\begin{equation}
\label{ZV2}
Z_V^{(2)}(q_1,q_2,\epsilon) 
= \sum_{n\geq 0} \epsilon^n\sum_{u\in V_{[n]}}Z_V^{(1)}(u,q_1)Z_V^{(1)}(\overline{u},q_2),
\end{equation}
where $u$ ranges over any $V_{[n]}$ basis, $Z_V^{(1)}(u,q_1)$ and $Z_V^{(1)}(\overline{u},q_2)$ are genus one 1-point functions,  and $\overline{u}$ is the dual of $u$ with respect to the square bracket Li-Zamolodchikov metric $\langle\, ,\rangle_{[\,]}$.
 We may similarly define a genus two partition function associated with any pair of $V$-modules $N_1$ and $N_2$:
\begin{equation}
\label{ZN1N2}
Z_{N_1,N_2}^{(2)}(q_1,q_2,\epsilon) = \sum_{n\geq 0} \epsilon^n\sum_{u\in V_{[n]}}Z_{N_1}^{(1)}(u,q_1)Z_{N_2}^{(1)}(\overline{u},q_2).
\end{equation}
The genus two partition function for the rank 1 Heisenberg VOA $M$ is described in detail in \cite{MT2}:
\begin{theorem}\label{g2_heisenthm}\cite{MT2}
For the rank 1 Heisenberg VOA $M$
\begin{equation}\label{g2_heisen}
Z_{M}^{(2)}(q_{1},q_{2},\epsilon )=\frac{1}{\eta(q_1)\eta(q_2)}\det(I-A_1 A_2)^{-1/2},
\end{equation}
for the Dedekind eta function $\eta(q)$. $Z_{M}^{(2)}$ is holomorphic on $\cal{D}^\epsilon$.
\end{theorem}
This result is generalized to a central charge $r$  Heisenberg VOA $M^r$ with two irreducible $M^r$-modules
$N^{r}_{\boldsymbol{\alpha}}=N_{\alpha_{1}}\otimes \ldots \otimes N_{\alpha_{r}}$ and $N^r_{\boldsymbol{\beta}}=N_{\beta_{1}}\otimes \ldots \otimes N_{\beta_{r}}$ (for irreducible $M$-modules $N_{\alpha_{i}},N_{\beta{i}}$) as follows. Defining $Z_{\boldsymbol{\alpha},\boldsymbol{\beta}}^{(2)}(q_{1},q_{2},\epsilon )
=Z_{N^{r}_{\boldsymbol{\alpha}},N^{r}_{\boldsymbol{\beta}}}^{(2)}(q_1,q_2,\epsilon) $
we obtain
\begin{theorem}\label{g2_heisenmodulethm}\cite{MT2}
\begin{equation}\label{g2_heisenmodule}
Z_{\boldsymbol{\alpha},\boldsymbol{\beta}}^{(2)}(q_{1},q_{2},\epsilon )=\left(Z_{M}^{(2)}(q_{1},q_{2},\epsilon )\right)^r
e^{i\pi (\boldsymbol{\alpha}.\boldsymbol{\alpha}\,\Omega_{11}+\boldsymbol{\beta}.\boldsymbol{\beta}\,\Omega_{22}+2\boldsymbol{\alpha}.\boldsymbol{\beta}\,\Omega_{12})},
\end{equation}
where $\boldsymbol{\alpha}.\boldsymbol{\beta}=\alpha_1\beta_1+\ldots +\alpha_r\beta_r$ etc.
\end{theorem}

\medskip

\subsection{The Genus One Degeneration of the Genus Two Partition Function}
\label{sec:degen} 
Of particular interest in this paper is the behaviour of $Z_V^{(2)}(q_1,q_2,\epsilon)$ in the degeneration limit $q_2\rightarrow 0$ where the surface degenerates to a torus with modular parameter $\tau$ of \eqref{eq:tau}. Naively, one might expect that $q_2^{C/24}Z_V^{(2)}(q_1,q_2,\epsilon)\rightarrow Z_V^{(1)}(q)$ for $q=e^{2\pi i \tau}$. However, this is not the case as the following result shows:
\begin{proposition}
For the Heisenberg VOA $M$, the genus one degeneration of the genus two partition function is given by
\begin{equation}
\lim_{q_2\rightarrow 0} q_2^{1/24}Z_M^{(2)}(q_1,q_2,\epsilon)= Z_M^{(1)}(q)\left[
1+{\frac {1}{576}} E_4(q_1) {\epsilon}^{4}+O\left(\epsilon^6\right) \right].
\end{equation}
\end{proposition}

\begin{proof}
Using \eqref{delE2}, \eqref{deleta} and \eqref{eq:tau} we have
\begin{align*}
\eta(q)&= \eta(q_1) + (\tau-\tau_1)\partial_{\tau_1}\eta(q_1)+ \frac{1}{2}(\tau-\tau_1)^2\partial^2_{\tau_1}\eta(q_1)+O\left(\epsilon^6\right)\\
&=\eta(q_1)\left[1+\frac {1}{24}E_{2}\left( q_{1}\right)  \epsilon^2
-\left(\frac {1}{1152}E_{2}(q_{1})^{2}+\frac {5}{576}E_{4}\left( q_{1} \right) 
\right)\epsilon^4  +O\left(\epsilon^6\right)\right].
\end{align*} 
so that 
\begin{align*}
Z_M^{(1)}(q)=\frac {1}{\eta(q)}
&=\frac {1}{\eta(q_1)}\left[1-\frac {1}{24}E_{2}\left( q_{1}\right)  \epsilon^2
+\left(\frac {1}{384}E_{2}(q_{1})^{2}+\frac {5}{576}E_{4}\left( q_{1} \right) 
\right)\epsilon^4  +O\left(\epsilon^6\right)\right].
\end{align*}
On the other hand, from \eqref{A20} and \eqref{g2_heisen} we find
\begin{align*}
\lim_{q_2\rightarrow 0} q_2^{1/24}Z_M^{(2)}(q_1,q_2,\epsilon)&=
\frac{1}{\eta(q_1)}\det(I-A_1 A_2(0))^{-1/2}\\
&= \frac{1}{\eta(q_1)}\left[1-\frac {1}{24}E_{2} \left( q_{{1}} \right) \epsilon^2
+ \left( 
{\frac {1}{384}}E_{{2}} ( q_{1})^2
+{\frac {1}{96}}E_{{4}} ( q_{1} ) \right)\epsilon^4+O\left(\epsilon^6\right)\right]\\
&=Z_M^{(1)}(q)\left[
1+{\frac {1}{576}} E_4(q_1) {\epsilon}^{4}+O\left(\epsilon^6\right) \right].
\end{align*}
\end{proof}

A similar result is shown in Proposition~7.5 in \cite{MT3} in the comparison between the Heisenberg genus two partition function $Z_M^{(2)}(q_1,q_2,\epsilon)$ and an alternative partition function  $Z_{M,\rho}^{(2)}(q,w,\rho)$  based on a different sewing scheme where a torus with modular parameter $q$ is self-sewn at punctures with relative position $w$ and with sewing parameter $\rho$  - see \cite{MT3} for the details.
There it is shown that the two partition functions are not equal (on the domains where the two sewing schemes can be compared). 
This apparent incompatibility is consistent with the conformal anomaly in physics where the partition function of a CFT is believed to describe a section of some bundle over moduli space so that the discrepancy is described by a non-trivial transition function e.g. \cite{FS}. Furthermore, it is believed in the physics literature that the transition functions are universal in the sense that the ratio of the partition functions for \textbf{any} two CFTs of a given central charge $C$ has a trivial section over moduli space. In particular, for positive integer $C$  we may compare the partition function of any given CFT to the $C$-dimensional bosonic string.  
Thus, in the language of VOAs, we can formulate the following conjecture \cite{MT3}. 
Let $V$ be a VOA with positive integer central charge $C$ and define genus two normalized partition functions, analogous to \eqref{thetaN}, by 
\begin{equation}
\Theta_V^{(2)}(q_1,q_2,\epsilon) \equiv \frac{Z_V^{(2)}(q_1,q_2,\epsilon) }{\left(Z_M^{(2)}(q_1,q_2,\epsilon)\right)^C},
\qquad 
\Theta_{V,\rho}^{(2)}(q,w,\rho) \equiv \frac{Z_{V,\rho}^{(2)}(q,w,\rho) }{\left(Z_{M,\rho}^{(2)}(q,w,\rho)\right)^C}
\label{Theta2}
\end{equation}
\begin{conjecture}
\label{conj}
$\Theta_V^{(2)}(q_1,q_2,\epsilon)=\Theta_{V,\rho}^{(2)}(q,w,\rho) $ on the domains where the two sewing schemes can be compared.
\end{conjecture}
For example, for an even  lattice VOA $V_L$, one finds $\Theta_V^{(2)}=\Theta_L^{(2)}(\Omega)$, the genus two Siegel lattice theta series, in both  sewing schemes \cite{MT2,MT3}.
\medskip

We now prove a non-trivial consistency condition related to the  torus degeneration limit $q_2\rightarrow 0$ of Conjecture~\ref{conj}. This corresponds to $\rho\rightarrow 0$ in the alternative sewing procedure where  $ \Theta_{V,\rho}^{(2)}(q,w,\rho) \rightarrow \Theta_V^{(1)}(q)$, the genus one normalized partition function.   Thus we wish to show:
\begin{theorem}\label{g2_mainresult}
Let $V$ be a VOA of  central charge $C$ with normalized genus two partition function $\Theta_V^{(2)}(q_1,q_2,\epsilon)$ which is convergent on a neighborhood of the two tori degeneration point $\epsilon=0$. Then under the torus degeneration
\begin{align*}
\lim_{q_2\rightarrow 0} \Theta_V^{(2)}(q_1,q_2,\epsilon)=  \Theta_V^{(1)}(q),
\end{align*}
where $q$ is the modular parameter for the degenerate torus.
\end{theorem}
\medskip
In order to prove Theorem~\ref{g2_mainresult} we first show the following:
\begin{lemma} 
\label{lem:degen}
Let $V$ be a VOA with central charge $C$ with genus two partition function $Z_V^{(2)}(q_1,q_2,\epsilon)$. Then
\begin{equation}
\lim_{q_2\rightarrow 0} q_2^{C/24} Z_V^{(2)}(q_1,q_2,\epsilon) = \sum_{n\geq 0} \epsilon^nZ_V^{(1)}(\tl^{[n]},q_1)\label{g2_limwithrho1}
\end{equation}
for $\tl^{[n]}\in V_{[n]}$ of \eqref{g2_tldefneqn}. 
\end{lemma}
\begin{proof}
Using Lemma~\ref{g2_kqbrackisolemma} we note that 
\begin{equation*}
\lim_{q_2\rightarrow 0} q_2^{C/24}Z_V^{(1)}(\overline{u},q_2)=\langle\textbf{1},o(\overline{u})\textbf{1}\rangle_{(\,)}=\langle\tl^{[n]},\overline{u}\rangle_{[\,]},
\end{equation*} 
for all $\overline{u}\in V_{[n]}$. It therefore follows from \eqref{ZV2} that
\begin{align*}
\lim_{q_2\rightarrow 0} q_2^{C/24} Z_V^{(2)}(q_1,q_2,\epsilon) 
&= \sum_{n\geq 0} \epsilon^n Z_V^{(1)}\left(\sum_{u\in V_{[n]}}\langle\tl^{[n]},\overline{u}\rangle_{[\,]} u,q_1\right)\notag\\
&= \sum_{n\geq 0} \epsilon^nZ_V^{(1)}(\tl^{[n]},q_1).
\end{align*}
\end{proof}
We now turn to the proof of Theorem~\ref{g2_mainresult}.
\begin{proof}
$\tl^{[n]}$ is a square bracket Virasoro descendent of even weight $\wtr[u]=n$ so that Corollary~\ref{g1_maintheoremcorol} implies
\begin{equation*}
Z_V^{(1)}(\tl^{[n]},q_1) = \frac{1}{\eta^{C}(q_1)}\sum_{i=0}^{\frac{n}{2}}\sum_{j=0}^{\frac{n}{2}-i} H_{ij}^n(q_1)C^j\qd_1^{i}\zh_V^{(1)}(q_1),
\end{equation*}
where   $\qd_1\equiv q_1\frac{\partial}{\partial q_1}=\frac{1}{2\pi i}\frac{\partial}{\partial \tau_1}$ and $H_{ij}^{n}(q_1)$ is a quasi-modular form of weight $n-2i$ independent of the VOA $V$.
(The upper limits in the $i,j$ sums arise from the $L[-2]^{n/2}\vv$ contribution to $ \tl^{[n]}$.) Thus Lemma~\ref{lem:degen} implies that
\begin{equation}
\label{eq:Hi}
\lim_{q_2\rightarrow 0} q_2^{C/24} Z_V^{(2)}(q_1,q_2,\epsilon)  = \frac{1}{\eta^{C}(q_1)}\sum_{l\ge 0}H_l(q_1,C,\epsilon)\qd_1^{l}\zh_V^{(1)}(q_1),
\end{equation}
where 
\begin{equation}
H_l(q_1,C,\epsilon)=\sum_{n\geq 2l} \sum_{j=0}^{\frac{n}{2}-l} H_{lj}^n(q_1)\epsilon^n C^j,
\label{eq:HlC}
\end{equation}
which is independent of the VOA up to the central charge $C$. In particular, we may compute 
$H_l(q_1,C,\epsilon)$ explicitly by considering the rank $r$ Heisenberg VOA genus two partition function \eqref{ZN1N2} for the pair of modules $N_1=N^{r}_{\boldsymbol{\alpha}}$ and $N_2=M^r$ i.e. $\boldsymbol{\beta} = 0$. Then Theorem~\ref{g2_heisenmodulethm} implies
\begin{equation*}
Z^{(2)}_{\boldsymbol{\alpha},0}(q_1,q_2,\epsilon) = \frac{1}{\eta^{r}(q_1)\eta^{r}(q_2)}\det(I-A_1A_2)^{-r/2}e^{i\pi\alpha^2\Omega_{11}(q_1,q_2,\epsilon)},
\end{equation*}
where $\alpha^2 = \boldsymbol{\alpha}.\boldsymbol{\alpha}$. Therefore, recalling that $\tau=\Omega_{11}(q_1,0,\epsilon)$,  we obtain
\begin{align}
\label{eq:Z2al}
\lim_{q_2\rightarrow0}q_2^{r/2}Z^{(2)}_{\boldsymbol{\alpha},0}(q_1,q_2,\epsilon) &= \frac{1}{\eta^{r}(q_1)}\det\left(I-A_1A_2(0)\right)^{-r/2}q^{\alpha^2/2}.
\end{align}
But, as for the proof of Lemma~\ref{lem:degen} and \eqref{eq:Hi}, we also have
\begin{align}
\lim_{q_2\rightarrow0}q_2^{r/2}Z^{(2)}_{\boldsymbol{\alpha},0}(q_1,q_2,\epsilon) &= \sum_{n\ge0}\epsilon^nZ_{N^r_{\boldsymbol{\alpha}}}^{(1)}(\tl^{[n]},q_1)\notag\\
&=\frac{1}{\eta^{r}(q_1)}\sum_{l\ge 0}H_l(q_1,r,\epsilon)\qd_1^{l}\zh_{N^r_{\boldsymbol{\alpha}}}^{(1)}(q_1)\notag\\
&=\frac{1}{\eta^{r}(q_1)}\sum_{l\ge 0}H_l(q_1,r,\epsilon)\left(\frac{\alpha^2}{2}\right)^{l}q_1^{\alpha^2/2},
\label{eq:Z2al2}
\end{align}
since $\zh_{N^r_{\boldsymbol{\alpha}}}^{(1)}(q_1) = q_1^{\alpha^2/2}$.
Comparing \eqref{eq:Z2al} and \eqref{eq:Z2al2} using 
$q^{\alpha^2/2} =e^{\pi i(\tau - \tau_1)\alpha^2}q_1^{\alpha^2/2}$
we conclude that 
\begin{equation}
H_l(q_1,r,\epsilon)=\det\left(I-A_1A_2(0)\right)^{-r/2}\,\frac{1}{l!}\left(2\pi i(\tau - \tau_1)\right)^l,
\label{eq:detHi}
\end{equation}
for all integers $r\ge 1$.
It follows from \eqref{eq:HlC} that  \eqref{eq:detHi}  holds for $H_l(q_1,C,\epsilon)$ with $r$ replaced by a general central charge $C$.
Finally, \eqref{eq:Hi} implies
\begin{align*}
\lim_{q_2\rightarrow 0} q_2^{C/24} Z_V^{(2)}(q_1,q_2,\epsilon)  &= \frac{1}{\eta^{C}(q_1)}\det\left(I-A_1A_2(0)\right)^{-C/2}
\sum_{l\ge 0}\frac{1}{l!}\left(\tau - \tau_1\right)^l\frac{\partial^l}{\partial \tau_1^l}\zh_V^{(1)}(q_1)\\
&=\frac{1}{\eta^{C}(q_1)}\det\left(I-A_1A_2(0)\right)^{-C/2}\zh_V^{(1)}(q),
\end{align*}
by Taylor's Theorem. The result follows on recalling  that
\begin{equation*}
\lim_{q_2\rightarrow 0} q_2^{C/24}\left(Z_M^{(2)}(q_1,q_2,\epsilon)\right)^C=
\frac{1}{\eta^C(q_1)}\det(I-A_1 A_2(0))^{-C/2}.
\end{equation*} 
\end{proof}

\section{Appendix. The Virasoro Vectors $\lambda^{(n)}$}
\label{Sect:App}
Recall from Section \ref{sec:g1_kqbracket} the exponential map
\begin{align*}
\phi(z)&=e^z-1=\exp\left(\sum_{i=0}^\infty \alpha_iz^{i+1}\der{z}\right)z,
\end{align*} 
where $ \alpha_1=\frac{1}{2}, \alpha_2=-\frac{1}{12}, \alpha_3=\frac{1}{48},\ldots$ are specific rational parameters. For the linear operator on $V$ 
\begin{equation*}
T = \exp\left(\sum_{i \ge 0} \alpha_i L_i\right),
\end{equation*}
we define the Virasoro vacuum descendents $\lambda^{(n)}\in V_n$  by
\begin{align*}
\lambda = \sum_{n\ge 0} \lambda^{(n)}=T^\dagger\vv =\exp\left(\sum_{i\ge0}\alpha_i L_{-i}\right)\vv.
\end{align*}
Here we show that
\begin{proposition}
\label{prop:lambda}
The Virasoro vacuum descendent $\lambda^{(n)}=0$ for odd $n$ and for even $n$ is determined by
\begin{displaymath}
\lambda = \ldots \exp\left(\beta_6 L_{-6}\right)\exp\left(\beta_4 L_{-4}\right)\exp\left(\beta_2 L_{-2}\right)\vv,
\end{displaymath}
for specific real parameters $\beta_{k}$ for all positive even $k$:
\begin{center}
\renewcommand{\arraystretch}{2}
\begin{tabular}{|c|c|c|c|c|c|c|c|}
\hline
$\beta_2$ & $\beta_4$ &$\beta_6$ & $\beta_8$ & $\beta_{10}$ & $\beta_{12}$ &$\beta_{14}$ &\ldots\\[0.3em]
\hline
$-\dfrac{1}{12}$ & $-{\dfrac {1}{480}}$ & ${\dfrac {1}{12096}}$ & $-{\dfrac {1}{138240}}$& 
	${\dfrac {1}{2280960}}$ & $-{\dfrac {389}{13586227200}}$&${\dfrac {1}{464486400}} $ &\ldots\\[0.6em]
\hline
\end{tabular}
\end{center}
\end{proposition}
Thus we find
\begin{align*}
\lambda^{(0)} &= \vv,\qquad \lambda^{(2)} = \beta_2 L_{-2}\vv= -\frac{1}{12}L_{-2}\vv,\\
\lambda^{(4)} &=
\frac{1}{2!}\beta_2^2 L^2_{-2}\vv +\beta_4 L_{-4}\vv=
 \frac{1}{288}L^2_{-2}\vv - \frac{1}{480}L_{-4}\vv,\\
\lambda^{(6)} &
= \frac{1}{3!}\beta_2^3 L^{3}_{-2}\vv 
+ \beta_2\beta_4 L_{-4}L_{-2}\vv 
+ \beta_6 L_{-6}\vv\\
&=
 -\frac{1}{10368}L^{3}_{-2}\vv + \frac{1}{5760}L_{-4}L_{-2}\vv + \frac{1}{12096}L_{-6}\vv.
\end{align*}
Note that $\lambda^{(2n)}=\frac{\beta_2^n}{n!}L_{-2}^n\vv+\ldots$ for all $n\ge 0$. 
We also remark that the $\beta_k$ parameters are very similar to the Bernoulli numbers $B_k/k!$ but the precise relationship between these numbers is not known to us.
\medskip

In order to prove Proposition~\ref{prop:lambda} we consider the exponential mapping $\phi(z)$ as a sequence of conformal maps generated by $z^{k+1}\der{z}$ for each $k\ge 1 $.  Defining
\begin{equation}
\label{wk}
w_k(z) = z\left(1-k\beta_k z^k\right)^{-\frac{1}{k}},
\end{equation}
for a complex parameter $\beta_k$ we find
\begin{lemma}\label{g2_expmap} 
For any power series $f(z)=\sum_{m\ge 0}a_m z^m$ we find
\begin{equation*}
\exp\left(\beta_k z^{k+1}\der{z}\right) f(z)=f(w_k(z)).
\end{equation*}
\end{lemma}
\begin{proof}
 A straightforward calculation shows that
\begin{align*}
\frac{1}{n!}\left(z^{k+1}\der{z}\right)^n \left(z^m\right)= z^m \binom{-\frac{m}{k}}{n}\left(-k z^k\right)^n,
\end{align*}
so that 
\begin{equation*}
\exp\left(\beta_k z^{k+1}\der{z}\right) z^m=\left(w_k(z)\right)^m.
\end{equation*} 
\end{proof}
We next consider the composition of conformal maps generated by $z^{k+1}\der{z}$ for all $k\ge 1$ that is equivalent to the exponential map. 
\begin{proposition}
\label{g2_expmapprop}
The exponential map $\phi(z)=e^z-1$ is equivalent to a sequence of conformal maps 
\begin{align*}
\phi(z) = \ldots \exp\left(\beta_3 z^{4}\der{z}\right) \exp\left(\beta_2 z^{3}\der{z}\right) 
\exp\left(\beta_1 z^{2}\der{z}\right) 
z,
\end{align*}
for specific rational parameters where $\beta_{2k+1}=0$ for all $k\ge 1$ and $\beta_1=\frac{1}{2}$, $\beta_2=-\frac{1}{12}$, $\beta_4=-\frac{1}{480},\beta_6=\frac{1}{12096},\ldots$
\end{proposition}
\begin{proof}
Repeatedly applying Lemma~\ref{g2_expmap} we find 
\begin{align*}
\phi(z) =w_1\left(w_2\left(w_3\left(\ldots z\ldots \right)\right)\right).
\end{align*}
Since 
\begin{equation*}
w_k(z)=z+\beta_k z^{k+1}+O(z^{2k+1}),
\end{equation*}
 we can iteratively solve for $\beta_k$. Thus $\phi(z)=w_1(z)+O(z^3)$ which implies $\beta_1=\frac{1}{2}$.
This process can be repeated to obtain each $\beta_{k}$. To solve for $\beta_2$ let
\begin{align}
u(z) &= w_2\left(w_3\left(w_4\left(\ldots z\ldots \right)\right)\right)=z+ \beta_2z^3 + O(z^4) \label{g2_uzexpans}.
\end{align}
But $ \phi(z) =w_1(u(z))=u(z)\left(1-\frac{1}{2}u(z)\right)^{-1}$ implies  
\begin{align}
u(z) &= 2\tanh\left(\frac{z}{2}\right)=z-\frac{1}{12}z^3 + O(z^4) \label{g2_uzistanh}.
\end{align}
Comparing to \eqref{g2_uzexpans} we thus obtain $\beta_2=-\frac{1}{12}$. 
Notice that  $u(z)$ is an odd function of $z$. But 
$w_{2k+1}(z) = z + \beta_{2k+1}z^{2k+2} + \ldots$
contributes an even power $z^{2k+2}$ to $u(z)$ in \eqref{g2_uzexpans} 
for all $k\ge 1$ so that $\beta_{2k+1}=0$. The even labeled parameters $\beta_4,\beta_6, \ldots$ can be computed by considering  the higher order terms in the expansion of $\phi(z)$.  
\end{proof}

We may now complete the proof of Proposition~\ref{prop:lambda}. Associated with each conformal map $w_k$ in \eqref{wk}, we define the linear operator
\begin{equation*}
T_k= \exp(\beta_k L_k).
\end{equation*}
Then  Proposition~\eqref{g2_expmapprop} implies that the linear operator $T$ associated with the exponential map $\phi$ can be written as
\begin{align*}
T &= T_{w_1}T_{w_2}T_{w_4}T_{w_6}\ldots= \exp\left(\beta_1 L_1\right)\exp\left(\beta_2 L_2\right)\exp\left(\beta_4 L_4\right)\exp\left(\beta_6 L_6\right)\ldots,
\end{align*}
for $\beta_k$ of  Proposition~\eqref{g2_expmapprop}. Furthermore,  the adjoint operator is 
\begin{equation*}
T^\dagger = \ldots \exp\left(\beta_6 L_{-6}\right)\exp\left(\beta_4 L_{-4}\right)\exp\left(\beta_2 L_{-2}\right)\exp\left(\beta_1 L_{-1}\right),
\end{equation*}
so that 
\begin{equation*}
\lambda =T^\dagger\vv=\ldots \exp\left(\beta_6 L_{-6}\right)\exp(\beta_4L_{-4})\exp(\beta_2L_{-2})\vv,
\end{equation*}
using $L_{-1}\vv=0$. \qed

\end{document}